\documentclass[10pt, a4paper]{amsart}

\usepackage{authblk}
\usepackage{enumitem}
\usepackage[charter]{mathdesign}
\usepackage{amsmath}                

\usepackage{amsfonts}
\usepackage{graphicx, subfigure}
\usepackage{graphics}
\usepackage{epstopdf}

\usepackage[usenames, dvipsnames, pdftex]{color}
\usepackage[colorlinks=true,
        raiselinks=true,
        linkcolor=MidnightBlue,
        citecolor=ForestGreen,
        urlcolor=Mahogany,
        plainpages=false]{hyperref}

\usepackage{tikz}
\usepackage[all]{xy}

    \numberwithin{equation}{section}

    \newcommand{\pr}{\mathsf{P}}

    \def\bm{\mathrm{b}\m}
    
    \def\Q{\mathbb{Q}}

    \def\P{\mathbb{P}}

    \def\R{\mathbb{R}}
    \def\N{\mathbb{N}}

    \def\p{\mathcal{P}}

    \def\d{\mathrm{d}}
    \def\d{\mathrm{d}}
    \def\f{\mathcal{F}}

    \def\q{\mathsf{Q}}
    
    \def\x{\mathfrak{X}}
    \def\y{\mathfrak{Y}}

    \def\m{\mathcal{M}}
    
    \def\b{\mathfrak{B}}

    \def\ve{\varepsilon}

\begin{document}


  {
      \theoremstyle{plain}
      \newtheorem{theorem}{Theorem}
      \newtheorem{lemma}{Lemma}
      \newtheorem{corollary}{Corollary}
      \newtheorem{scheme}{Scheme}
      \newtheorem{proposition}{Proposition}
      \newtheorem{assumption}{Assumption}
      \newtheorem{example}{Example}
      \newtheorem{remark}{Remark}
      \newtheorem{definition}{Definition}
  }


\title[Perturbation of conditional probabilities in total variation]{On the effect of perturbation of conditional probabilities in total variation}

\maketitle

\author{Alessandro Abate}
\address{Department of Computer Science, University of Oxford}
\email{Alessandro.Abate@cs.ox.ac.uk}

\author{Frank Redig}
\address{Delft Institute of Applied Mathematics, TU Delft}
\email{f.h.j.redig@tudelft.nl}

\author{Ilya Tkachev}
\address{Delft Center for Systems and Control, TU Delft}
\email{i.tkachev@tudelft.nl}

\begin{abstract}
  A celebrated result by A. Ionescu Tulcea provides a construction of a probability measure on a product space given a sequence of regular conditional probabilities.
  We study how the perturbations of the latter in the total variation metric affect the resulting product probability measure.
\end{abstract}

\section{Introduction}

The Ionescu Tulcea extension theorem \cite[Section 2.7.2]{a1972} states that given a sequence of stochastic kernels,
there exists a unique probability measure on the product space generated by this sequence,
that is a measure whose conditional probabilities equal to these kernels.
Such a construction is often used in the theory of general Markov Decision Processes \cite{bs1978},
and general Markov Chains \cite{r1984} in particular.
Hence, it is of a certain interest to study how sensitive the resulting product measure is with respect to perturbations of the generating sequence of kernels.
A possible direct application of such result concerns numerical methods,
where characteristics of the original stochastic process are studied over its simpler approximations,
often defined over a finite state space.
Such approximations can be further regarded as a perturbation of the original sequence of kernels \cite{ta2013} which connects to the original problem.

Here we specifically focus on the metric between kernels and measures given by the total variation norm.
Given the pairwise distances between corresponding transition kernels in this metric,
we are interested in bounds on the distance between the resulting product measures.
A similar study was given in \cite{rr2013} which used the Borel assumption,
that is it assumed that spaces involved are (standard) Borel spaces.
However, the bounds obtained in \cite{rr2013} grow linearly with the cardinality of the sequence and hence are not tight:
recall that the total variation distance between two probability measures is always bounded from above by $2$.

In this paper we provide two results:
the first generalizes linear bounds in \cite{rr2013} to the case of arbitrary measurable spaces,
whereas the second result shows that under the Borel assumption it is possible to derive much sharper bounds,
that appear to be precise in some special cases --
e.g. in case of independent products of measures.

The rest of the paper is structured as follows.
Section \ref{sec:prob} gives a problem formulation together with statements of main results.
Proofs are given in Section \ref{sec:proofs},
which is followed by the discussion in Section \ref{sec:discus} and an enlightening example in Section \ref{sec:example}.
The notation is provided in Appendix, Section \ref{sec:appendix}.

\section{Problem statement}
\label{sec:prob}

Let us recall the construction of the product measure given the regular conditional probabilities.
First of all, we need the following notion of a product of a probability measure and a stochastic kernel which extends a more usual product of two measures.

\begin{proposition}\label{prop:product}
  Let $(X,\x)$ and $(Y,\y)$ be arbitrary measurable spaces.
  For any probability measure $\mu\in \p(X,\x)$ and any stochastic kernel $K:X\to \p(Y,\y)$ there exists a unique probability measure $\q\in \p(X\times Y,\x\otimes \y)$,
  denoted by $\q := \mu\otimes K$,
  such that
  \begin{equation*}
    \q(A\times B) = \int_A K(x,B)\mu(\d x)
  \end{equation*}
  for any pair of sets $A\in \x$ and $B\in \y$.
\end{proposition}

\begin{proof}
  For a proof, see \cite[Section 2.6.2]{a1972}.
\end{proof}

The construction above immediately extends to any finite sequence of spaces by induction,
whereas for the countable products the following result holds true.

\begin{proposition}[Ionescu-Tulcea]\label{prop:I-T}
  Let $\{(X_k,\x_k)\}_{k \in \N_0}$ be a family of arbitary measurable spaces and let $(\Omega_n,\f_n) = \prod_{k=0}^n(X_k,\x_k)$ be product spaces for any $n\in \bar\N_0$.
  For any probability measure $P^0\in \p(X_0,\x_0)$ and any sequence of stochastic kernels $(P^k)_{k\in \N}$,
  where $P^k:\Omega_{k-1}\to \p(X_k,\x_k)$,
  there exists a unique probability measure $\pr \in \p(\Omega_\infty,\f_\infty)$,
  denoted by $\pr := \bigotimes_{k=0}^\infty P^k$,
  such that the finite-dimensional marginal $\pr^n$ of $\pr$ on the measurable space $(\Omega_n,\f_n)$ is given by $\pr^n  = \bigotimes_{k=0}^n P^k$ for any $n\in \N_0$.
\end{proposition}

\begin{proof}
  For a proof, see \cite[Section 2.7.2]{a1972}.
\end{proof}

In the setting of Proposition \ref{prop:I-T},
suppose that we are given another sequence of kernels $(\tilde P^k)_{k=0}^\infty$ and let $\tilde\pr := \bigotimes_{k=0}^\infty \tilde P^k$ be the corresponding product measure.
Given the assumption that $\|P^k - \tilde P^k\| \leq c_k$ for any $k\in \N_0$ and some sequence of reals $(c_k)_{k\in \N_0}$,
we study how the distance $\|\pr^n - \tilde \pr^n\|$ can be bounded.
For the general case of arbitrary measurable spaces,
the following result holds true.

\begin{theorem}\label{thm:lin}
  Let $\{(X_k,\x_k)\}_{k \in \N_0}$ be any family of measurable spaces and let $\tilde \x_k\subseteq \x_k$ for any $k\in \N_0$.
  Denote by $(\Omega_n,\f_n) = \prod_{k=0}^n (X_k,\x_k)$ and $(\Omega_n,\tilde \f_n) = \prod_{k=0}^n (X_k,\tilde \x_k)$ the corresponding product spaces for any $n\in \bar\N_0$.
  Let $P^0 \in \p(X_0,\x_0)$, $\tilde P^0\in \p(X_0,\tilde \x_0)$ and let kernels $P^k :\Omega_{k-1} \to \p(X_k,\x_k)$ and $\tilde P^k:\Omega_{k-1}\to \p(X_k,\tilde \x_k)$ for $k\in \N$ be $\f_{k-1}$- and $\tilde \f_{k-1}$-measurable respectively.
  If a sequence of reals $(c_k)_{k\in \N_0}$ is such that $\|P^k  - \tilde P^k\|\leq c_k$ for all $k\in \N_0$,
  then for any $n\in \N_0$ it holds that
  \begin{equation}\label{eq:thm.lin}
    \|\pr^n - \tilde\pr^n\|\leq \sum_{k=0}^n c_k.
  \end{equation}
\end{theorem}

\begin{remark}\label{rem:conv}
  Through this paper,
  and in particular in the statement of Theorem \ref{thm:lin},
  we use the following convention.
  If the domain of one measure is a subset of the domain of another,
  the total variation distance between them is taken over the smaller domain.
  For example,
  in the setting of Theorem \ref{thm:lin} we have $\|P^0 - \tilde P^0\| = 2\cdot \sup_{A\in \tilde \x_0}|P^0(A) - \tilde P^0(A)|$.
\end{remark}

As it has been mentioned in Introduction,
the bounds \eqref{eq:thm.lin} are not tight.
For example, if $c_k = c>0$ for all $k\in \N_0$ then the right-hand side of \eqref{eq:thm.lin} is $c\cdot n$ and diverges to infinity as $n\to\infty$,
whereas the left-hand side stays bounded above by $2$.
It appears,
that under a rather mild assumption that all involved measurable spaces are (standard) Borel,
a stronger result can be obtained.

\begin{theorem}\label{thm:main}
  Let $\{X_k\}_{k \in \N_0}$ be a family of Borel spaces and let $\Omega_n = \prod_{k=0}^n X_k$ be product spaces for any $n\in \bar\N_0$.
  Let further $P^0,\tilde P^0\in \p(X_0)$ and $P^n,\tilde P^k:\Omega_{k-1}\to \p(X_k)$ for $k\in \N$.
  If a sequence of reals $(c_k)_{n\in \N_0}$ is such that $\|P^k  - \tilde P^k\|\leq c_k$ for all $k\in \N_0$,
  then for any $n\in \N_0$ it holds that
  \begin{equation}\label{eq:thm.main}
    \|\pr^n - \tilde\pr^n\|\leq 2 - 2\prod_{k=0}^n\left(1-\frac12c_k\right).
  \end{equation}
\end{theorem}

The proofs of both theorems are given in the next section.

\section{Proofs of the main results}
\label{sec:proofs}

\subsection{Proof of Theorem \ref{thm:lin}}
\label{ssec:proof.lin}
We prove both theorems by induction,
by first studying how the perturbation of a measure and a kernel in Proposition \ref{prop:product} propagates to the product measure.
The following lemma provides such study in the setting of Theorem \ref{thm:lin}.

\begin{lemma}\label{lem:lin}
  Let $(X,\x)$ and $(Y,\y)$ be two measurable spaces,
  and let $\tilde \x \subseteq \x$ and $\tilde \y\subseteq \y$.
  Consider $\mu\in \p(X,\x)$ and $\tilde\mu\in \p(X,\tilde \x)$,
  and suppose that kernels $K:X\to \p(Y,\y)$ and $\tilde K:X\to \p(Y,\tilde \y)$ are $\x$- and $\tilde \x$-measurable,
  respectively.
  Denote by $\q := \mu\otimes K$ and $\tilde \q := \tilde\mu\otimes \tilde K$ the corresponding product measures.
  It holds that
  \begin{equation}\label{eq:lem.lin}
    \|\q - \tilde \q\|\leq \|\mu - \tilde\mu\| + \|K - \tilde K\|
  \end{equation}
  where in \eqref{eq:lem.lin} we follow the convention in Remark \ref{rem:conv}.
\end{lemma}

\begin{proof}
  Let the set $A\in \tilde \x\otimes \tilde\y$ be arbitrary,
  and denote by $A_x = \{y\in Y:(x,y)\in A\}$ the $x$-section of $A$ for any $x\in X$.
  It follows from \cite[Section 2.6.2]{a1972} that
  \begin{equation*}
    \q(A) = \int_X K_x(A_x)\mu(\d x),\qquad \tilde\q(A) = \int_X \tilde K_x(A_x)\tilde \mu(\d x)
  \end{equation*}
  and as a result:
  \begin{align*}
    |\q(A) - \tilde \q(A)| &= \left|\int_X K_x(A_x)\mu(\d x) - \int_X \tilde K_x(A_x)\tilde \mu(\d x)\right|
    \\
    &\leq \left|\int_X \left(K_x(A_x) - \tilde K_x(A_x)\right)\mu(\d x)\right| + \left|\int_X \tilde K_x(A_x)(\mu - \tilde \mu)(\d x)\right|
    \\
    &\leq \sup_{x\in X}\sup_{B\in \tilde \y}\left|K_x(B) - \tilde K_x(B)\right| + \sup_{B\in \tilde \x}|\mu(B) - \tilde \mu(B)|
    \\
    &= \frac12\left(\|\mu - \tilde\mu\| + \|K - \tilde K\|\right)
  \end{align*}
  which together with \eqref{eq:tv.min} implies \eqref{eq:lem.lin}.
\end{proof}

To prove Theorem \ref{thm:lin} we are only left to apply the result of Lemma \ref{lem:lin} by induction to $\mu = \pr^n$,
$\tilde \mu = \tilde \pr^n$,
$K = P^{n+1}$ and $\tilde K = \tilde P^{n+1}$ for $n\in \N_0$.

\subsection{Proof of Theorem \ref{thm:main}}

Again, we are going to apply induction,
however in the current case the analogue of Lemma \ref{lem:lin} requires a more intricate proof via the coupling techniques.
Let us briefly recall some facts about the coupling.
Given an arbitrary measurable space $(\Omega,\f)$,
a coupling of two probability measures $\pr,\tilde \pr \in \p(\Omega,\f)$ is a probability measure $\P\in \p(\Omega^2,\f^2)$ such that the marginals of $\P$ are given by
\begin{equation}\label{eq:coup}
  \pi_*\P = \pr,\qquad \tilde\pi_*\P = \tilde \pr,
\end{equation}
where $\pi(\omega,\tilde \omega) = \omega$ and $\tilde \pi(\omega,\tilde\omega) = \tilde \omega$ for all $(\omega,\tilde\omega)\in \Omega^2$ \cite[Section 1.1]{l1992}.
In particular, if $\Omega$ is a Borel space and $\f = \b(\Omega)$,
we have the following result:
\begin{equation}\label{eq:coup.in}
  \P(\pi = \tilde \pi) \leq \|\pr\wedge\tilde\pr\|,
\end{equation}
The inequality \eqref{eq:coup.in} is called the coupling inequality \cite[Section 1.2]{l1992};
it holds true for any coupling measure $\P$ as per \eqref{eq:coup}.
At the same time,
thanks to the fact that $\Omega$ is a Borel space,
there always exists a maximal coupling:
the one for which the equality holds in \eqref{eq:coup.in}.
As for some pairs of probability measures there may be several choices of their maximal coupling,
here we focus on the $\gamma$-coupling \cite[Section 1.5]{l1992}.

\begin{definition}[$\gamma$-coupling]
  Let $Z$ be a Borel space and let $\nu, \tilde \nu \in \p(Z)$ be two probability measures on it.
  The $\gamma$-coupling of $(\nu,\tilde \nu)$ is a measure $\gamma \in \p(Z^2)$ given by
  \begin{equation*}
    \gamma(\nu,\tilde \nu):= (\psi_Z)_*(\nu\wedge\tilde \nu) + 1_{[0,1)}(\|\nu\wedge\tilde\nu\|)\cdot\frac{(\nu - \tilde \nu)^+\otimes (\nu-\tilde\nu)^-}{1 - \|\nu\wedge\tilde\nu\|}
  \end{equation*}
  where $\psi_Z:Z\to Z^2$ is the diagonal map on $Z$ given by $\psi_Z:z\mapsto (z,z)$.
\end{definition}

The following lemma is the key in the proof of Theorem \ref{thm:main}.

\begin{lemma}\label{lem:main.1}
  Let $X$ and $Y$ be two Borel spaces,
  and let $\mu,\tilde \mu\in \p(X)$ and $K,\tilde K:X\to \p(Y)$.
  Denote by $\q := \mu\otimes K$ and $\tilde \q := \tilde\mu\otimes \tilde K$ the corresponding product measures.
  It holds that
  \begin{equation}\label{eq:lem.main.1}
    \|\q \wedge \tilde \q\|\geq \|\mu \wedge \tilde \mu\|\cdot\inf_{x\in X}\|K_x \wedge \tilde K_x\|.
  \end{equation}
\end{lemma}

\begin{proof}
  The proof is done via coupling of measures $\q$ and $\tilde \q$ in a sequential way.

  Firstly, let $m = \gamma(\mu,\tilde\mu) \in \p(X^2)$ be the $\gamma$-coupling of $(\mu,\tilde\mu)$.
  Secondly, we define a kernel $\kappa:X^2\to \p(Y^2)$ by the following formula:
  \begin{equation*}
    \kappa_{x\tilde x} = 1_{\Delta_X}(x,\tilde x)\cdot\gamma\left(K_x,\tilde K_{\tilde x}\right) + 1_{\Delta_X^c}(x,\tilde x)\cdot (K_x\otimes \tilde K_{\tilde x})
  \end{equation*}
  where $\Delta_X$ is the diagonal of $X^2$ as per Section \ref{sec:appendix}.
  Clearly, the measure $\kappa_{x\tilde x}$ is a coupling of measures $K_x$ and $\tilde K_{\tilde x}$ for any $(x,\tilde x)\in X^2$,
  which is a maximal coupling on the diagonal,
  and the independent (or product) coupling off the diagonal.
  Note that $\kappa:X^2\to \p(Y^2)$ is indeed a kernel.
  The only non-trivial part of the latter statement concerns the measurability of the positive part and the negative part in the Hahn-Jordan decomposition of the kernel,
  which follows directly from \cite[Lemma 1.5, Chapter 6]{r1984}.
  Furthermore, the measurability of $K_x\otimes \tilde K_{\tilde x}$ obviously holds for any measurable rectangle $A\times \tilde A\subseteq Y^2$ and extends to the whole product $\sigma$-algebra by the monotone class theorem (see e.g. \cite[Theorem 1.3.9]{a1972}).

  Define $\Q: = m\otimes \kappa \in \p(X^2\times Y^2)$ to be the product measure,
  and further denote by $\pi_X$, $\tilde \pi_X$, $\pi_Y$ and $\tilde \pi_Y$ the obvious projection maps from that space, e.g.
  \begin{equation*}
    \tilde\pi_X(x,\tilde x,y,\tilde y) = \tilde x\in X.
  \end{equation*}
  We claim that the random element $(\pi_X,\pi_Y)$ is distributed according to $\q$ and $(\tilde \pi_X,\tilde \pi_Y)$ is distributed according to $\tilde \q$. Indeed, for any $A\in \b(X)$ and $B\in \b(Y)$ it holds that
  \begin{align*}
    \Q(\pi_X\in A,\pi_Y\in B) & = \Q((A\times X)\times (B\times Y)) &\text{ by definition of $\pi_X$ and $\pi_Y$}
    \\
    &= \int_{A\times X}\kappa_{x\tilde x}(B\times Y)m(\d x\times \d \tilde x) &\text{ by Proposition \ref{prop:product}}
    \\
    & = \int_{A\times X}K_x(B)m(\d x \times \d\tilde x) &\text{ since $\kappa_{x\tilde x}$ is a coupling}
    \\
    &= \int_A K_x(B)\mu(\d x) &\text{ since $\mu$ is a marginal of $m$}
    \\
    &= \q(A\times B). &\text{ by Proposition \ref{prop:product}}
  \end{align*}
  Since both $(\pi_X,\pi_Y)_*\Q$ and $\q$ are probability measures on a Borel space $X\times Y$,
  and they have been shown to agree on the class measurable rectangles which is closed under finite intersections,
  they are equal \cite[Proposition 3.6, Chapter 0]{r1984}.
  Similarly, it holds that $(\tilde \pi_X, \tilde \pi_Y)_*\Q = \tilde \q$.
  Thus, by the coupling inequality \eqref{eq:coup.in}
  \begin{equation*}
    \|\q \wedge \tilde \q\|\geq \Q(\pi_X = \tilde \pi_X,\pi_Y = \tilde \pi_Y)
  \end{equation*}
  On the other hand, for the latter term the following holds:
  \begin{align*}
    \Q(\pi_X = \tilde \pi_X,\pi_Y = \tilde \pi_Y) & = \Q((\Delta_X\times Y^2)\cap (X^2\times\Delta_Y)) = \Q(\Delta_X\times \Delta_Y)
    \\
    & = \int_{\Delta_X}\kappa_{x\tilde x}(\Delta_Y)m(\d x\times \d\tilde x)
    \\
    & \geq m(\Delta_X)\cdot \inf_{(x,\tilde x)\in \Delta_X}\kappa_{x\tilde x}(\Delta_Y) = m(\Delta_X)\cdot \inf_{x\in X}\kappa_{xx}(\Delta_Y)
    \\
    & = \|\mu\wedge\tilde\mu\|\cdot\inf_{x\in X}\|K_x \wedge \tilde K_x\|.
  \end{align*}
  and the lemma is proved.
\end{proof}

As we have mentioned above,
Lemma \ref{lem:main.1} is an analogue of Lemma \ref{lem:lin}.
However,
unlike in Section \ref{ssec:proof.lin},
here we cannot go directly from Lemma \ref{lem:main.1} to Theorem \ref{thm:main} and we need the following auxiliary result first:

\begin{lemma}\label{lem:main.2}
  Let $\{X_k\}_{k \in \N_0}$ be a family of Borel spaces and let $\Omega_n = \prod_{k=0}^n X_k$ be product spaces for any $n\in \bar\N_0$.
  Let further $P^0,\tilde P^0\in \p(X_0)$ and $P^n,\tilde P^n:\Omega_{n-1}\to \p(X_n)$ for $n\in \N$ be initial probability measures and conditional stochastic kernels respectively.
  Denote:
  \begin{equation*}
    a_0 := \|P^0 \wedge \tilde P^0\|,\quad a_k := \inf_{\omega_{k-1}\in \Omega_{k-1}}\|P^k_{\omega_{k-1}} \wedge \tilde P^k_{\omega_{k-1}}\|,\quad k\in \N.
  \end{equation*}
  and further $\pr^n := \bigotimes_{k=0}^n P^k$, $\tilde\pr^n := \bigotimes_{k=0}^n \tilde P^k$.
  For any $n\in \N_0$ it holds that
  \begin{equation}\label{eq:lem.main.2}
    \|\pr^n \wedge \tilde \pr^n\|\geq \prod_{k=0}^n a_k.
  \end{equation}
\end{lemma}

\begin{proof}
  The inequality \eqref{eq:lem.main.2} can be proved by induction.
  Clearly, it holds true for $n = 0$.
  Suppose it holds true for some $n\in \N_0$,
  then in the setting of Lemma \ref{lem:main.1} put $X = \Omega_n$, $Y = X_{n+1}$, $\mu = \pr^n$, $\tilde \mu = \tilde \pr^n$, $K = P^{n+1}$ and $\tilde K = \tilde P^{n+1}$.
  For the product measures we obtain that $\q = \pr^{n+1}$ and $\tilde \q = \tilde \pr^{n+1}$,
  so \eqref{eq:lem.main.2} now follows immediately from \eqref{eq:lem.main.1}.
\end{proof}

To prove Theorem \ref{thm:main} we are only left to notice that \eqref{eq:thm.main} is equivalent to \eqref{eq:lem.main.2} thanks to the duality argument in \eqref{eq:tv.min}.

\section{Discussion}
\label{sec:discus}

Let us discuss the results obtained above and their connection to the literature.
First of all, Theorem \ref{thm:lin} extends the bounds on the total variation distance between the finite-dimensional marginals $\pr^n$ and $\tilde\pr^n$,
obtained in \cite[Theorem 1]{rr2013} under the Borel assumption,
to the case of arbitrary measurable spaces.
Its proof is inspired by the one of \cite[Lemma 1]{ta2013} that focused on the Markovian case $P^1 = P^2 = \dots$ exclusively.
The work in \cite{ta2013} also emphasized the benefit of dealing with sub-$\sigma$-algebras as in Theorem \ref{thm:lin}:
the perturbation constants $c_k$ in such case are likely to be smaller than those in Theorem \ref{thm:main}.
Moreover,
with focus on numerical methods,
it allows dealing with kernels that may not have an integral expression.
Although the bounds obtained in \cite{rr2013} is a special case of Theorem \ref{thm:lin},
the main focus of \cite{rr2013} was rather on the construction of the corresponding maximal coupling of the infinite-dimensional product measures $\pr$ and $\tilde \pr$.
In the setting of Theorem \ref{thm:lin} the existence of such coupling is unlikely,
as even the measurability of the diagonal,
required in the coupling inequality,
may be violated in the case of arbitrary measurable spaces.

At the same time,
in the setting of \cite{rr2013} (that is, under the Borel assumption) a stronger result in Theorem \ref{thm:main} holds true.
Although the latter theorem is only focused on the bounds,
Lemma \ref{lem:main.1} which is the core in the proof of Theorem \ref{thm:main},
yields the coupling of finite-dimensional marginals $\pr^n$ and $\tilde\pr^n$ that trivially extends to the infinite-dimensional case by Proposition \ref{prop:I-T}.
With focus on bounds,
to show that Theorem \ref{thm:main} indeed provides a less conservative estimate than \cite{rr2013} (or Theorem \ref{thm:lin}),
let us mention that
\begin{equation}\label{eq:bound.comparison}
  2 - 2\prod_{k=0}^n\left(1-\frac12c_k\right) \leq \sum_{k=0}^n c_k
\end{equation}
for any non-negative sequence $(c_k)_{k\in \N_0}$ and any $n\in \N_0$,
which can be shown by induction over $n$.\footnote{
  Notice also that the right-hand side of \eqref{eq:bound.comparison} can be considered as the first-order term of the expansion of the product in the left-hand side.
}
In particular,
there are several cases when bounds in Theorem \ref{thm:main} are exact:
e.g. when kernels $P^k$ and $\tilde P^k$ are just measures on $X_k$ which corresponds to dealing with a sequence of iid random variables,
see also the example in Section \ref{sec:example}.
In contrast, the bounds in Theorem \ref{thm:lin} can grow unboundedly e.g. in case $c_k = c>0$ for all $k\in \N_0$.

Finally,
let us mention that it is of interest whether Theorem \ref{thm:main} holds true in the general case of arbitrary measurable spaces and arbitrary stochastic kernels --
recall however that even in such case the corresponding maximal coupling may not exist.
Although we do not provide the answer to this general question,
we show that under the assumption that kernels have densities,
one can obtain a version of Lemma \ref{lem:main.1} for the case of arbitrary measurable spaces:
this of course implies the validity of Theorem \ref{thm:main} in such case as well,
provided that all the kernels involved have densities.

\begin{lemma}\label{lem:int}
  Let $(X,\x)$ and $(Y,\y)$ be two arbitrary measurable spaces,
  and let measures $\mu,\tilde \mu\in \p(X,\x)$ and $\lambda \in\p(Y,\y)$.
  If kernels $K, \tilde K:X\to \p(Y,\y)$ are given by
  \begin{equation*}
    K_x(\d y) := k(x,y)\lambda(\d y),\qquad \tilde K_x(\d y) := \tilde k(x,y) \lambda(\d y)
  \end{equation*}
  where $k,\tilde k:X\times Y \to [0,\infty)$ are $\x\otimes \y$-measurable functions,
  then for the product measures $\q := \mu\otimes K$ and $\tilde \q := \tilde\mu\otimes \tilde K$ the inequality \eqref{eq:lem.main.1} holds true.
\end{lemma}

\begin{proof}
  Note that the following expressions for $\q$ and $\tilde \q$ hold true:
  \begin{align*}
    \q(\d x\times \d y) &= k(x,y) [\mu\otimes \lambda](\d x\times \d y),
    \\
    \tilde\q(\d x\times \d y) &= \tilde k(x,y) [\tilde \mu\otimes \lambda](\d x\times \d y).
  \end{align*}
  To make them comparable,
  let us define $\nu := \mu+\tilde\mu$,
  so that
  \begin{align*}
    \|\q \wedge \tilde \q\| &= (\q\wedge \tilde \q)(X\times Y)
    \\
    &= \int_{X\times Y}\min\left(k(x,y)\frac{\d \mu}{\d \nu}(x),\tilde k(x,y)\frac{\d \tilde \mu}{\d \nu}(x)\right)[\nu\otimes\lambda](\d x\times \d y)
    \\
    &\geq \int_{X\times Y}\min\left(k(x,y),\tilde k(x,y)\right)\min\left(\frac{\d \mu}{\d \nu}(x),\frac{\d \tilde \mu}{\d \nu}(x)\right)[\nu\otimes\lambda](\d x\times \d y)
    \\
    &= \int_X (K_x\wedge \tilde K_x)(Y) (\mu\wedge\tilde\mu)(\d x) \geq \|\mu\wedge\tilde \mu\|\cdot \inf_{x\in X}\|K_x \wedge \tilde K_x\|
  \end{align*}
  as desired.
\end{proof}

\section{Example}
\label{sec:example}

To enlighten the theoretical results obtained above,
we study how conservative are the bounds in Theorem \ref{thm:main} on a simple example for which the analytical expression for the total variation distance is available.
For that purpose, we consider a Markov Chain with just two states,
that is $X = \{0,1\}$ and $P^k = P$ for all $k\in \N$.
The kernel $P$ is given by a stochastic matrix,
which for simplicity we assume to be diagonal: $P = \left(\begin{smallmatrix} 1&0\\ 0&1 \end{smallmatrix} \right)$.
The initial distribution,
denoted by $P^0 = \mu$,
is further given by $\mu(\{1\}) = p$, $\mu(\{0\}) = 1-p$.
We focus on the case when the perturbed stochastic process is a Markov Chain as well,
that is $\tilde P^k = \tilde P$ for all $k\in \N$,
and assume that $\tilde P = \left(\begin{smallmatrix} 1-\ve&\ve\\ \ve&1-\ve \end{smallmatrix} \right)$ for some $\ve\in (0,1)$.
The perturbed initial distribution we denote by $\tilde P^0 = \tilde \mu$,
it is given by $\mu(\{1\}) = p-\delta$ and $\mu(\{0\}) = 1-(p-\delta)$;
here $\delta\in (p-1,p)$.

Since the product space $\Omega_n = \{0,1\}^{n+1}$ is finite,
the precise value of the total variation distance between the product measures $\pr^n$ and $\tilde \pr^n$ can be computed as
\begin{equation*}
  \left\|\pr^n - \tilde\pr^n\right\| = \sum_{\omega_n\in\Omega_n} \left|\pr^n(\{\omega_n\})  - \tilde\pr^n(\{\omega_n\})\right|
\end{equation*}
Thanks to the special choice of $P$,
the measure $\pr^n$ is supported only on two points in $\Omega_n$ which makes it is easy to obtain an analytic expression for the sum above.
Let us abbreviate: $f_n(\ve):= 1 - (1-\ve)^n$.
There are three possible cases depending on the interplay between parameters $p$, $\ve$, $\delta$ and $n$:
\begin{itemize}
  \item[a.] If $\delta < - \frac{pf_n(\ve)}{1-f_n(\ve)}$,
  then $\|\pr^n - \tilde\pr^n\| = 2(1-p)f_n(\ve) - 2\delta(1-f_n(\ve))$.
  \item[b.] If $- \frac{pf_n(\ve)}{1-f_n(\ve)}\leq \delta \leq \frac{(1-p)f_n(\ve)}{1-f_n(\ve)}$,
  then $\|\pr^n - \tilde\pr^n\| = 2f_n(\ve)$.
  \item[c.] If $\delta > \frac{(1-p)f_n(\ve)}{1-f_n(\ve)}$,
  then $\|\pr^n - \tilde\pr^n\| =  2pf_n(\ve) + 2\delta(1-f_n(\ve))$.
\end{itemize}
According to Theorem \ref{thm:main},
the bounds are $\|\pr^n - \tilde\pr^n\|\leq 2(f_n(\ve) + \delta(1-f_n(\ve)))$.
Clearly, in all cases [a.], [b.] and [c.] the bounds hold true.
Moreover, if only the stochastic matrix is perturbed,
that is $\delta = 0$ (case [b.]),
then the total variation is $2f_n(\ve)$ which precisely coincides with the bounds provided by Theorem \ref{thm:main}.
Similarly,
if $p=1$ and $\delta\in (0,1)$ (case [c.]),
then the total variation $2(f_n(\ve) + \delta(1-f_n(\ve)))$ provides another example when bounds in Theorem \ref{thm:main} are exact.
In all other cases it is easy to see that there is a strictly positive gap in the inequality \eqref{eq:thm.main}.

Unrelated to the precision of bounds in Theorem \ref{thm:main},
it is interesting to further comment on the phenomenon in case [b.].
If $p$, $\ve$ and $\delta$ are fixed,
then there exists $N\in \N$ such that for $n\geq N$ always [b.] holds:
indeed, it follows directly from the fact that the denominator $1-f_n(\ve) = (1-\ve)^n\to 0$ monotonically as $n\to\infty$.
As a result, regardless of the value of $\delta\in (0,1)$,
for all $n$ big enough $\|\pr^n - \tilde\pr^n\|$ does not depend on $\delta$:
in particular, it is the same as in the case when $\delta = 0$.

\section*{Acknowledgements}

The last author would like to thank George Lowther for the idea of non-linear bounds,
Michael Greinecker for his hints on dealing with kernels,
and other users of \href{http://math.stackexchange.com/}{MathStackexchange} and \href{http://mathoverflow.net/}{MathOverflow} for their valuable comments regarding measure theory and coupling.

\bibliographystyle{amsalpha}
\bibliography{../my_bib}

\section{Appendix}
\label{sec:appendix}

For any set $X$ the corresponding diagonal is denoted by $\Delta_X := \{(x,x):x\in X\} \subseteq X^2$.
The set of all real numbers is denoted by $\R$ and the set of all natural numbers is denoted by $\N$.
We further write $\N_0 := \N\cup \{0\}$ and $\bar\N_0 := \N_0\cup\{\infty\}$.
For any $A\subseteq X$ we denote its indicator function by $1_A$.

We say that $X$ is a (standard) Borel space if $X$ is a topological space homeomorphic to a Borel subset of a complete separable metric space.
Any Borel space is assumed to be endowed with its Borel $\sigma$-algebra $\b(X)$.
An example of a Borel space is the space of real numbers $\R$ endowed with the Euclidean topology.

Given a measurable space $(X,\x)$ we denote the space of all $\sigma$-additive finite signed measures on it by $\bm(X,\x)$.
For any set $A\in\x$ we introduce an evaluation map $\theta_A:\bm(X,\x) \to \R$ given by $\theta_A(\mu) := \mu(A)$.
We always assume that $\bm(X,\x)$ is endowed with the smallest $\sigma$-algebra that makes all evaluation maps measurable.
The subspace of all probability measures in $\bm(X,\x)$ is denoted by $\p(X,\x)$.
In case $X$ is a Borel space,
we simply write $\bm(X)$ and $\p(X)$ in place of $\bm(X,\b(X))$ and $\p(X,\b(X))$ respectively.
If $(Y,\y)$ is another measurable space,
by a bounded kernel we mean a measurable map $K:X\to \bm(Y,\y)$ such that
\begin{equation*}
  \sup_{x\in X}\sup_{A\in \y}|K_x(A)|<\infty
\end{equation*}
where we write $K_x$ instead of a more cumbersome $K(x) \in \bm(Y,\y)$ for any $x\in X$.
In case $K_x \in \p(Y,\y)$ for any $x\in X$,
we say that the kernel $K$ is stochastic.
The condition that $K:X\to \bm(Y,\y)$ is measurable is equivalent to $K_{(\cdot)}(A):X\to \R$ being a measurable map for any $A\in \y$ \cite[Lemma 1.37]{k1997a}.
Clearly, any measure can be considered as a kernel which does not depend on its first argument.
Furthermore, any measure $\nu\in \bm(X,\x)$ admits the unique Hahn-Jordan decomposition given by $\nu = \nu^+ - \nu^-$,
where $\nu^+$ and $\nu^-$ are two mutually singular non-negative measures on $(X,\y)$ \cite[Section 2.1.2]{a1972}.
If $K:X\to \bm(Y,\y)$ is a kernel,
then
\begin{equation*}
  \|K\|:= \sup_{x\in X}\left(K^+_x(Y)+K^-_x(Y)\right)
\end{equation*}
defines the total variation of $K$.
For any two measures $\nu,\tilde\nu\in \bm(X,\x)$ we denote their minimum by $\nu\wedge\tilde\nu := \nu - (\nu - \tilde \nu)^- \in \bm(X,\x)$.
In particular, it holds that
\begin{equation}\label{eq:tv.min}
  \|\nu \wedge \tilde \nu\| = 1 - \frac12 \|\nu - \tilde \nu\| = 1 - \sup_{A\in \x}|\nu(A) - \tilde \nu(A)|,
\end{equation}
for any two probability measures $\nu,\tilde \nu\in \p(X,\x)$.

If $f:X\to Y$ is a measurable map between two measurable spaces $(X,\x)$ and $(Y,\y)$,
for any $\nu\in \bm(X,\x)$ the image measure $f_*\nu \in \bm(Y,\y)$ is given by
\begin{equation*}
  (f_*\nu)(A) := \nu\left(f^{-1}(A)\right)
\end{equation*}
for all $A\in \y$.
Given a family of measurable spaces $\{(X_k,\x_k)\}_{k\in \N_0}$ and an index set $I\subseteq \bar\N_0$,
we denote the product measurable space by
\begin{equation*}
  \prod_{k\in I}(X_k,\x_k) := \left(\prod_{k\in I} X_k,\bigotimes_{k\in I}\x_k\right),
\end{equation*}
where $\bigotimes_{k\in I}\x_k$ is the product $\sigma$-algebra.
Let $\tilde I\subseteq I$ and denote $(\Omega,\f) := \prod_{k\in I}(X_k,\x_k)$ and $(\tilde\Omega,\tilde\f) := \prod_{k\in \tilde I}(X_k,\x_k)$.
Let further $\pi: \Omega \to \tilde \Omega$ be an obvious projection map.
For any $\nu \in \bm(\Omega,\f)$ we say that $\pi_*\mu$ is the marginal of $\mu$ on $(\tilde \Omega,\tilde \f)$.

\end{document}